\documentclass[usenames,dvipsnames, 11pt]{article}
\usepackage[utf8]{inputenc}

\usepackage[margin=2.5cm,a4paper]{geometry} 
\usepackage{multirow,listings,setspace,gnuplottex,latexsym,keyval,ifthen,moreverb,lscape,forest}
\usepackage{pgf,tikz}
\usetikzlibrary{arrows}

\usepackage{amsmath, amssymb, amsthm, bm}
\usepackage{thm-restate}
\usepackage{hyperref}
\usepackage[capitalize]{cleveref}
\usepackage{color}
\usepackage{url}
\usepackage{enumitem}
\usepackage{graphicx}
\usepackage{appendix}
\usepackage{cite}
\usepackage{mathtools}

\newtheorem{thm}{Theorem}
\newtheorem{prop}[thm]{Proposition}

\newtheorem{cor}[thm]{Corollary}
\newtheorem{clm}[thm]{Claim}

\newtheorem{ques}[thm]{Question}

\numberwithin{thm}{section}

\theoremstyle{definition}
\newtheorem{mydef}{Definition}

\newcommand{\cov}[1][k]{\operatorname{cov}_{#1}}

\DeclarePairedDelimiter{\parens}{(}{)}
\DeclarePairedDelimiter{\set}{\{}{\}}
\DeclarePairedDelimiter{\card}{|}{|}
\DeclarePairedDelimiter{\floor}{\lfloor}{\rfloor}
\DeclarePairedDelimiter{\ceil}{\lceil}{\rceil}
\DeclarePairedDelimiter{\brackets}{[}{]}

\DeclarePairedDelimiter{\size}{|}{|}

\usepackage{diagbox}

\def\cD{{\mathcal D}}
\def\cP{{\mathcal P}}
\def\cI{{\mathcal I}}
\def\cL{{\mathcal L}}
\def\cH{{\mathcal H}}

\title{Covering grids with multiplicity}
\author{Anurag Bishnoi\;\thanks{Delft Institute of Applied Mathematics, Technische Universiteit Delft, 2628 CD Delft, Netherlands. E-mail: \textsf{A.Bishnoi@tudelft.nl}.} \and Simona Boyadzhiyska\thanks{School of Mathematics, University of Birmingham, Edgbaston, Birmingham, B15 2TT, UK. This research was done when the author was affiliated with the Institut f\"ur Mathematik, Freie Universit\"at Berlin, 14195 Berlin, Germany. E-mail: \textsf{s.s.boyadzhiyska@bham.ac.uk}. Research supported by the Deutsche Forschungsgemeinschaft (DFG, German Research 
Foundation) under Germany's Excellence Strategy – The Berlin Mathematics 
Research Center MATH+ (EXC-2046/1, project ID: 390685689, BMS Stipend) and Graduiertenkolleg ``Facets of Complexity'' (GRK 2434).} \and Shagnik Das\;\thanks{Department of Mathematics, National Taiwan University, Taiwan. E-mail: \textsf{shagnik@ntu.edu.tw}. Research supported by Taiwan NSTC grant 111-2115-M-002-009-MY2.} \and Yvonne den Bakker\;\thanks{Delft Institute of Applied Mathematics, Technische Universiteit Delft, 2628 CD Delft, Netherlands. E-mail: \textsf{ym.den.bakker@gmail.com }. }}

\begin{document}
\maketitle

\begin{abstract}
Given a finite grid in $\mathbb{R}^2$, how many lines are needed to cover all but one point at least $k$ times? Problems of this nature have been studied for decades, with a general lower bound having been established by Ball and Serra. We solve this problem for various types of grids, in particular showing the tightness of the Ball--Serra bound when one side is much larger than the other. In other cases, we prove new lower bounds that improve upon Ball--Serra and provide an asymptotic answer for almost all grids. For the standard grid $\{0,\hdots,n-1\} \times \{0,\hdots,n-1\}$, we prove nontrivial upper and lower bounds on the number of lines needed. To prove our results, we combine linear programming duality with some combinatorial arguments.
\end{abstract}

\section{Introduction}

A celebrated result of Alon and F\"uredi~\cite{AF93} in combinatorial geometry states that any multiset of hyperplanes that covers all but one point of a $d$-dimensional finite grid $S_1 \times \cdots \times S_d \subseteq \mathbb{F}^d$ over an arbitrary field $\mathbb{F}$ must have size at least $\sum_{i = 1}^d (|S_i| - 1)$. This lower bound is easily seen to be tight by taking all hyperplanes of the form $x_i - a = 0$ for $1 \leq i \leq d$ and $a \in S_i \setminus \{b_i\}$, where $(b_1, \dots, b_n)$ is the point that is uncovered. This is a significant theorem for a few different reasons; not only did the proof of Alon and F\"uredi play an important role in the development of the polynomial method~\cite{Alon99, Guth_book, BCPSch18}, but this result and its generalisations have also seen several applications in a wide variety of mathematical disciplines~\cite{K94, BBSz10, BCPSch18, ABR22, BGGSZ22}.

One such generalisation that has been studied by several researchers is the multiplicity version of the problem, where the points of the grid should be covered multiple times. We introduce some notation to define this problem formally.

\begin{mydef}
Given finite subsets $S_1, S_2, \hdots, S_d$ of some field $\mathbb{F}$, we write $\Gamma = \Gamma(S_1, S_2, \hdots, S_d)$ for the grid $S_1 \times S_2 \times \hdots \times S_d \subseteq \mathbb{F}^d$. Note that by translation we may, and will, assume $\vec{0} \in \Gamma$. We call a point in $\Gamma$ a \emph{boundary point} if any of its coordinates is equal to $0$, and an \emph{interior point} otherwise.

For a given integer $k \ge 1$, we call a multiset $\mathcal{H}$ of hyperplanes in $\mathbb{F}^d$ a \emph{$k$-cover of $\Gamma$} if every nonzero point of $\Gamma$ is contained in at least $k$ of the hyperplanes, while $\vec{0}$ is not covered at all. We denote by $\cov(\Gamma; \mathbb{F})$ the minimum cardinality of a $k$-cover of $\Gamma$ in $\mathbb{F}^d$. In the case $\mathbb{F} = \mathbb{R}$, we shall omit the field from the notation and simply write $\cov(\Gamma)$.
\end{mydef}

In this notation, the Alon--F\"uredi Theorem establishes that $\cov[1](\Gamma,\mathbb{F}) = \sum_{i=1}^d \parens*{\size*{S_i} - 1 }$ for any grid $\Gamma$ over any field $\mathbb{F}$. The multiplicity extension asks for the value of $\cov(\Gamma; \mathbb{F})$ for multiplicities $k \ge 2$, and we can start with a few trivial observations. First, if we remove any hyperplane from a $k$-cover, what we are left with is still a $(k-1)$-cover, and so $\cov(\Gamma;\mathbb{F}) \ge \cov[k-1](\Gamma;\mathbb{F}) + 1$; that is, this extremal function is strictly increasing in $k$. In the other direction, since the union of a $k$-cover and an $\ell$-cover yields a $(k + \ell)$-cover, we have $\cov[k+\ell](\Gamma; \mathbb{F}) \le \cov(\Gamma; \mathbb{F}) + \cov[\ell](\Gamma; \mathbb{F})$, and so the function is subadditive in $k$. Applying these recursive inequalities repeatedly until we reach the $k=1$ case of Alon--F\"uredi, we have
\begin{equation} \label{eq:trivial_bounds}
\sum_{i=1}^d \parens*{\size*{S_i} - 1} + k-1 = \cov[1](\Gamma; \mathbb{F}) + k-1 \le \cov(\Gamma; \mathbb{F}) \le k \cov[1](\Gamma; \mathbb{F}) = k \sum_{i=1}^d \parens*{\size*{S_i} - 1}.
\end{equation}

The goal, then, is to narrow the considerable gap between these bounds, and there has been much previous research on some specific cases. Predating the work of Alon and F\"uredi~\cite{AF93}, the study of affine blocking sets in finite geometry corresponds to setting $\mathbb{F} = \mathbb{F}_q$ for some prime power $q$ and taking $\Gamma = \mathbb{F}_q^d$. For this grid, the classic paper of Jamison~\cite{J77} uses the polynomial method to prove $\cov[1](\Gamma; \mathbb{F}_q) = d(q-1)$. Bruen~\cite{B92} later used the polynomial method with multiplicities to provide lower bounds for the multiplicity version, showing $\cov(\Gamma, \mathbb{F}_q) \ge (d+k-1)(q-1)$.
This is an improvement upon~\eqref{eq:trivial_bounds}, but is generally not tight~\cite{Z02, LR14}. In~\cite{BBDM22}, the first three authors together with Tam\'as M\'esz\'aros obtained new bounds in the case $q=2$ by exploiting an equivalence between $k$-covers and linear binary codes of minimum distance $k$.

Recent work of Clifton and Huang~\cite{CH20} considered this problem over $\mathbb{R}$, where the grid is the hypercube $\Gamma = \{0,1\}^d$. For fixed dimension $d$ and growing multiplicity $k$, they used linear programming to determine $\cov(\Gamma)$ asymptotically. On the other hand, when the dimension $d$ is large with respect to the multiplicity $k$, they applied the polynomial method to provide general lower bounds that are tight for $k = 2$ and $k=3$. However, they conjectured that their lower bound of $d + k + 1$ is not tight for $k \ge 4$, and that the true value of $\cov(\Gamma)$ is $d + \binom{k}{2}$ for all fixed $k$ and large enough $d$ (see \cite[Conjecture 4.1]{CH20}). A subsequent paper of Sauermann and Wigderson~\cite{SW20} determined the best bound one can obtain with the polynomial method, where one seeks the minimum possible degree of a polynomial that does not vanish at the origin but has zeroes of multiplicity $k$ at all other points in the grid. While their result improves the Clifton--Huang lower bound to $\cov(\Gamma) \geq d + 2k - 3$, it still falls short of the conjectured value of $\cov(\Gamma)$ in this case, suggesting a strong separation between the algebraic and geometric problems.

Apart from these special cases, Ball and Serra~\cite[Theorem 5.3]{BS09} applied the polynomial method to obtain a lower bound valid for any grid $\Gamma = \Gamma(S_1, S_2, \hdots, S_d)$ over any field $\mathbb{F}$:
\begin{equation}\label{eq:ball-serra-bound}
    \cov(\Gamma; \mathbb{F}) \geq \sum_{i=1}^d \parens{\card{S_i}-1} + (k-1)\max\limits_{1 \le i \le d}\parens{\card{S_i}-1}.
\end{equation}
This extends the bound of Bruen, and is a sizeable improvement on the lower bound of~\eqref{eq:trivial_bounds}, but remains far removed from the upper bound. Indeed, in the symmetric case when $\size{S_i} = n$ for all $i \in [d]$, we have $(d+k-1)(n-1) \le \cov(\Gamma; \mathbb{F}) \le kd(n-1)$. It is thus of great interest to determine whether the Ball--Serra bound can be tight, and to obtain better bounds when it is not.

\medskip 

In this paper we initiate the systematic study of the covering problem for two-dimensional finite grids over $\mathbb{R}$. Our first set of results concerns how the dimensions of the grid affect the tightness of the Ball--Serra bound. We start by showing that when the grid is much wider than it is tall, the Ball--Serra bound is sharp.

\begin{thm}
\label{thm:asymmetric2d}
Let $S_1, S_2 \subseteq \mathbb{R}$ satisfy $\size{S_1} = n$, $\size{S_2} = m$, and $0 \in S_1 \cap S_2$, and set $\Gamma = \Gamma(S_1, S_2) \subseteq \mathbb{R}^2$. If, for a positive integer $k$, we have $n \ge (k-1)(m-1) + 1$, then
\[ \cov(\Gamma) = k(n-1) + (m-1). \]
\end{thm}

Our next result shows that the lower bound on $n$ from~\cref{thm:asymmetric2d} cannot be improved in general. In fact, when $n \le (k-1)(m-1)$, the Ball--Serra bound can be improved for almost all $n \times m$ grids. We make the ``almost all" precise with the following definition.

\begin{mydef}
Let $S_1, S_2 \subseteq \mathbb{R}$ with $0 \in S_1 \cap S_2$, let $\Gamma = \Gamma(S_1, S_2)$, and let $\Delta \ge 0$ be an integer. We call $\Gamma$ \emph{$\Delta$-generic} if any line containing two boundary points contains at most $\Delta$ interior points. In the case $\Delta = 0$, when such lines avoid all interior points, we simply call $\Gamma$ \emph{generic}.
\end{mydef}

To see that grids are typically generic, suppose the nonzero points of $S_1$ and $S_2$ are sampled uniformly and independently from $[-1,1]$. If a line contains two boundary points and an interior point, then the boundary points must come from different axes, and so we may assume the points in question are $(0,b_1)$, $(a_1,0)$ and $(a_2, b_2)$ for some nonzero $a_1, a_2 \in S_1$ and $b_1, b_2 \in S_2$. For these points to be collinear, we require $a_1(b_1 - b_2) = a_2 b_1$, and the probability of such an equation holding for any choice of $a_i, b_j$ is $0$. We are now ready to state our next result, which concerns $\cov(\Gamma)$ in the special case where $\Gamma$ is a generic grid. 

\begin{thm}
\label{thm:tightboundsasymmetric}
Let $\Gamma = \Gamma(S_1,S_2) \subseteq \mathbb{R}^2$ be a generic grid with $|S_1| = n$, $|S_2| = m$ and $0 \in S_1 \cap S_2$. Then 
\[ \cov(\Gamma) \ge k(n-1) + \frac{k}{n + m - 2} (m-1)^2.\]
Furthermore, if $\Gamma$ is an arbitrary $n \times m$ grid, and if $k$ is divisible by $\frac{n+m-2}{\gcd(n-1,m-1)}$, then we have
\begin{equation}\label{eq:tightboundasym}
     \cov(\Gamma) \le k(n-1) + \frac{k}{n+m-2}(m-1)^2.
\end{equation}
In particular, if $\Gamma$ is generic, we have equality above.
\end{thm}

Note that $\frac{k}{n+m-2}(m-1)^2$ is strictly larger than $m-1$ precisely when $n \le (k-1)(m-1)$, and hence these theorems show the Ball--Serra bound is tight for generic grids if and only if $n \ge (k-1)(m-1)+1$.

\medskip

Theorem~\ref{thm:tightboundsasymmetric} shows that the Ball--Serra bound gets worse as the dimensions of the grid grow closer. For the majority of our paper, therefore, we focus on square grids $\Gamma$, where $n = m$. In this case, the Ball--Serra bound gives $\cov(\Gamma) \ge (k+1)(n-1)$. Our next theorem gives a stark improvement for general square grids.

\begin{restatable}{thm}{squaregridsgeneral}
\label{thm:squaregridsgeneral}
    Let $\Gamma = \Gamma(S_1, S_2) \subset \mathbb{R}^2$ be a grid with $|S_1| = |S_2| = n$ and $0 \in S_1 \cap S_2$. Then, for any integer $k \ge 2$, we have:
    \begin{enumerate}[label=(\alph*)]
        \item $\cov(\Gamma) \leq \ceil*{\frac{3}{2}k}(n-1)$.\label{thm:squaregridsub}
        \item $\cov(\Gamma) \ge (10 - 4 \sqrt{5} + o(1))k(n-1)$, where the asymptotics are for $n\to\infty$.\label{thm:squaregridslb}
        \item if $\Gamma$ is $\Delta$-generic for some $\Delta \ge 0$, then $\cov(\Gamma) \ge \brackets*{2 - \frac{n-1}{2(n-1) - \Delta}} k(n-1)$. \label{thm:squaregridsgeneric}
    \end{enumerate}
\end{restatable}

Note that part~\ref{thm:squaregridsub} improves the trivial upper bound from~\eqref{eq:trivial_bounds}, while, since $10 - 4 \sqrt{5} \approx 1.0557$, part~\ref{thm:squaregridslb} gives a constant factor improvement over the Ball--Serra bound in~\eqref{eq:ball-serra-bound}, showing that it is never tight for large square grids. Moreover, as previously discussed, almost all grids are generic, and substituting $\Delta = 0$ into part~\ref{thm:squaregridsgeneric} gives a lower bound that matches the upper bound of part~\ref{thm:squaregridsub} exactly when $k$ is even and asymptotically when $k$ is odd and large. In fact, part~\ref{thm:squaregridsgeneric} gives an asymptotically tight bound for all $o(n)$-generic grids.

\medskip

However, the most natural grid to consider is the \emph{standard $n\times n$ grid}, given by $S_1 = S_2 =\set{0,1,2,\dots, n-1}$, and we denote this by $\Gamma_n = \Gamma(S_1, S_2)$. By considering the diagonal $x+y=n-1$, we see that $\Gamma_n$ is not $\Delta$-generic for any $\Delta < n-2$, and so Theorem~\ref{thm:squaregridsgeneral}\ref{thm:squaregridsgeneric} is worse than the Ball--Serra bound when $n > k$. By tailoring our methods to this specific grid, we obtain the following improvements on the general bounds.

\begin{thm}\label{thm:standardgridbounds}
    Let $n,k\geq 2$ be integers, let $S=\set{0,1,2,\dots, n-1}$ and let $\Gamma_n = \Gamma(S,S)$ be the standard $n \times n$ grid. Then, as $n,k\to \infty$, we have
    \begin{align} \label{eq:standardgridbounds}
        (2 - e^{-1/2} +o(1))k(n-1) \leq \cov(\Gamma_n) \leq (\sqrt{2}+o(1))k(n-1).
    \end{align}
\end{thm}

Note that $2 - e^{-1/2} \approx 1.3935$, while $\sqrt{2} \approx 1.4142$, so there is still a gap between the best lower and upper bounds we obtain for standard grids. To obtain sharper bounds, it can help to restrict the class of $k$-covers we consider. As we will see in the proof of the upper bound in~\cref{thm:standardgridbounds}, our construction uses only three types of lines: horizontal, vertical, and lines of slope $-1$. A subsequent computer search verified that for small values of $n$ and $k$, we can always find an optimal $k$-cover using only these three kinds of lines. In our final result, we provide a matching lower bound for these restricted $k$-covers, suggesting the upper bound of~\eqref{eq:standardgridbounds} may be correct in the unrestricted case as well.

\begin{thm}\label{thm:lowerboundthreetypes}
    As $n \to \infty$, the smallest $k$-cover of the standard grid $\Gamma_n$ that only contains lines of slope $0, \infty$, or $-1$ has size at least $\parens*{\sqrt{2}+o(1)}k(n-1)$.
\end{thm}

When proving these results, we shall establish the upper bounds by means of explicit constructions of $k$-covers. The lower bounds, meanwhile, will follow by applying duality to a linear programming relaxation of this problem, and we shall set up this framework in Section~\ref{sec:linearprogramming}. Having described our methodology, we will prove Theorems~\ref{thm:asymmetric2d} and~\ref{thm:tightboundsasymmetric} in Section~\ref{sec:ballserra}. We then shift our focus to square grids, proving Theorems~\ref{thm:squaregridsgeneral},~\ref{thm:standardgridbounds}, and~\ref{thm:lowerboundthreetypes} in Section~\ref{sec:squaregrids}. Finally, we provide some concluding remarks and open problems in Section~\ref{sec:conclusion}.

\section{The linear programming framework} \label{sec:linearprogramming}

In this section we introduce the linear programming method that will be used to prove our lower bounds. The use of linear programming in extremal combinatorics is well-established and has led to many results (see, for example,~\cite{GMS2022}), including the Clifton--Huang~\cite{CH20} lower bound on $\cov\parens{\Gamma}$ for $\Gamma = \{0,1\}^d$ in the case when $d$ is fixed and $k$ tends to infinity. The standard template is as follows: assume without loss of generality that we are working with a minimization problem; first, we interpret our extremal problem as an instance of an integer programming problem. Then, since integer programming is intractable, we consider the linear programming relaxation, where we allow fractional solutions. Since an integer solution is in particular a fractional solution, the value of the linear program gives a lower bound to the original extremal problem. Crucially, we can then apply duality, and the task of finding lower bounds translates to finding feasible solutions to the dual linear program. We refer the reader to~\cite{matousek2006understanding} for further background on linear programming.

\medskip

To start with this plan of action, we need to recast the covering problem as an integer linear program. Given a grid $\Gamma = \Gamma(S_1, S_2)$, for some finite sets $S_1, S_2 \subseteq \mathbb{R}$ containing $0$, we wish to determine $\cov\parens{\Gamma}$, the minimum number of lines in a $k$-cover of $\Gamma$. We shall assume $\card{S_1}, \card{S_2} \ge 2$, as the problem is otherwise trivial.

For every possible line $\ell$ not containing the origin, we introduce a variable $z(\ell)$ indicating its multiplicity in the $k$-cover. We thus wish to minimise $\sum_{\ell} z(\ell)$, the size of the cover. In order to ensure that the solution returned by the integer program is a $k$-cover, we need to require that each nonzero point of $\Gamma$ be covered at least $k$ times while the origin be omitted altogether. The latter constraint is easily seen to be satisfied. Hence, for each $(x,y) \in \Gamma \setminus \{(0,0)\}$, we require that the sum of $z(\ell)$ over all origin-avoiding lines $\ell$ containing $(x,y)$ be at least $k$.

The one catch is that there are infinitely many lines in $\mathbb{R}^2$. To obtain a finite program, we observe that we may restrict our attention to lines that contain at least two nonzero points in $\Gamma$. Indeed, in any minimal $k$-cover of $\Gamma$, every line must contain at least one nonzero point, and if a line contains only one point $(x,y)$, then we can replace it by a different origin-avoiding line passing through $(x,y)$ and at least one other point of $\Gamma$. Thus, we need only consider lines from the set $\mathcal{L} = \mathcal{L}(\Gamma)$ of origin-avoiding lines containing at least two points of $\Gamma$. Since there are fewer than $\size{\Gamma}^2$ pairs of nonzero points in $\Gamma$, each of which determines a unique line, it follows that $\mathcal{L}$ is finite.

In summary, $\cov(\Gamma)$ is the solution to the following integer linear program $\mathcal{I} = \mathcal{I}(\Gamma, k)$.
\begin{align*}
    &\text{minimize} \quad \sum\limits_{\ell\in \cL} z(\ell)\\
    &\text{subject to}\notag\\
    &\quad\quad\quad\quad\sum\limits_{\substack{\ell\in \cL:\\ (x,y)\in \ell}} z(\ell) \geq k \quad \text{for all }(x,y)\in \Gamma \setminus \{(0,0)\} \\
    &\quad\quad\quad\quad z(\ell) \in \mathbb{Z}_{\geq 0} \quad \text{for all }\ell\in \cL
\end{align*}

The final constraint, that the variables $z(\ell)$ be integral, renders solving the program computationally infeasible. Instead, to obtain a polynomial-time solvable problem, we can relax the variables to be real-valued. Now that we are no longer constrained to the integers, we can also divide through by $k$, removing the dependence of the program on this parameter. We therefore obtain the linear program $\cP = \cP(\Gamma)$ with the normalised variables $u(\ell)$ for $\ell \in \cL$.
\begin{align} \label{eq:primal}
    &\text{minimize} \quad\sum\limits_{\ell\in \mathcal{L}} u(\ell) \\
    &\text{subject to}\notag\\
    &\quad\quad\quad\quad \sum\limits_{\substack{\ell\in \mathcal{L}: \\ (x,y)\in \ell}} u(\ell) \geq 1 \quad \text{for all }(x,y)\in \Gamma \setminus \{(0,0)\} \notag \\
    &\quad\quad\quad\quad u(\ell) \geq 0 \quad \text{for all }\ell\in \mathcal{L} \notag
\end{align}

Let us denote by $\Phi(\Gamma)$ the solution to $\cP(\Gamma)$. The following result shows that $\Phi(\Gamma)$ describes the asymptotic behaviour of $\cov(\Gamma)$ when $k$ is large with respect to the dimensions of the grid.

\begin{prop} \label{prop:linprogbound}
For any grid $\Gamma$ and integer $k \ge 1$ we have
\[ k \Phi(\Gamma) \le \cov(\Gamma) \le k \Phi(\Gamma) + \size{\cL}. \]
\end{prop}

\begin{proof}
As previously established, $\cov(\Gamma)$ is the value of the integer linear program $\cI(\Gamma,k)$. If we let $(z(\ell) : \ell \in \cL)$ be a solution to the program, then setting $u(\ell) = \tfrac{1}{k} z(\ell)$ yields a feasible solution to the linear relaxation $\cP(\Gamma)$, with value $\sum_{\ell} u(\ell) = \tfrac{1}{k} \sum_{\ell} z(\ell) = \tfrac{1}{k} \cov(\Gamma)$. As $\Phi(\Gamma)$ is the minimum possible value of a feasible solution to $\cP(\Gamma)$, the first inequality follows.

For the upper bound, let $\parens*{u^*(\ell) : \ell \in \cL}$ be an optimal solution to the linear program $\cP(\Gamma)$, with value $\Phi(\Gamma)$. If we then set $z(\ell) = \ceil*{ku^*(\ell)}$ for all $\ell \in \cL$, we obtain a feasible solution to $\cI(\Gamma, k)$. Thus,
\[ \cov(\Gamma) \le \sum_{\ell} z(\ell) = \sum_{\ell} \ceil*{ku^*(\ell)} \le \sum_{\ell} \parens*{k u^*(\ell) + 1} = k \Phi(\Gamma) + \size{\cL}. \qedhere \]
\end{proof}

We remark that in practice one often obtains better error bounds; for instance, the $\size{\cL}$ term can be replaced by the size of the support of the optimal solution $u^*$. Furthermore, for an infinite sequence of multiplicities $k$, we can do away with the error term altogether. Indeed, since all the coefficients of $\cP(\Gamma)$ are integral, there is a rational optimal solution $u^*$ (the one returned by the Simplex Algorithm, for example). If $k$ is divisible by the common divisor of the fractions $u^*(\ell)$, then we can set $z(\ell) = k u^*(\ell)$ without needing any rounding, thereby obtaining a solution to $\cI(\Gamma,k)$ of value precisely $k \Phi(\Gamma)$.

\medskip

Therefore, asymptotically as $k$ tends to infinity, the problem reduces to determining $\Phi(\Gamma)$. We can provide upper bounds by finding feasible solutions to $\cP(\Gamma)$ or, better yet, constructing $k$-covers of $\Gamma$. To obtain lower bounds, we appeal to the theory of duality. The dual of $\cP(\Gamma)$, which we denote by $\cD(\Gamma)$, is the following linear program, where we have a variable $w(x,y)$ for each point $(x,y)\in \Gamma \setminus \{(0,0)\}$, which we call the weight of the point.
\begin{align}\label{eq:dual}
    &\text{maximize} \quad \sum\limits_{(x,y)\in \Gamma \setminus \{(0,0)\}} w(x,y)\\
    &\text{subject to}\notag\\
    &\quad\quad\quad\quad \sum\limits_{\substack{(x,y)\in \Gamma \setminus \{(0,0)\}:\\ (x,y)\in \ell}} w(x,y) \leq 1 \quad \text{for all }\ell\in \mathcal{L}\notag\\
    &\quad\quad\quad\quad w(x,y) \geq 0 \quad \text{for all }(x,y)\in \Gamma \setminus \{(0,0)\} \notag
\end{align}

For convenience, given a set $S \subseteq \Gamma \setminus \{(0,0)\}$, we write $w(S) = \sum_{(x,y) \in S} w(x,y)$ for the weight of $S$. The dual program thus asks for the maximum possible weight of the grid, provided every line in $\cL$ has weight at most $1$. By the duality theorem for linear programming (see~\cite[Section 6.1]{matousek2006understanding}), the programs $\cP(\Gamma)$ and $\cD(\Gamma)$ have the same optimal objective value $\Phi(\Gamma)$. We shall thus prove our lower bounds on $\cov(\Gamma)$ by finding suitably large feasible weights on the grid $\Gamma$.

\section{Wide rectangular grids}
\label{sec:ballserra}

In this section we will prove Theorems~\ref{thm:asymmetric2d} and~\ref{thm:tightboundsasymmetric}, establishing precisely when the Ball--Serra bound is tight for all grids of given dimensions. Our first result establishes the value of $\cov(\Gamma)$ for all $n \times m$ grids $\Gamma$ whenever $n \ge (k-1)(m-1) + 1$.

\begin{proof}[Proof of~\cref{thm:asymmetric2d}]
The Ball--Serra bound~\eqref{eq:ball-serra-bound} provides the requisite lower bound, as substituting $\size{S_1} = n$ and $\size{S_2} = m$ gives $\cov(\Gamma) \ge k(n-1) + m-1$. To prove a matching upper bound, we provide an explicit construction of a $k$-cover containing this many lines.

Write $S_2 = \set{0,t_1,\dots, t_{m-1}}$, and let $P_{1}\cup \dots \cup P_{m-1}$ be an arbitrary partition of $S_1\setminus\set{0}$ such that $\card{P_i}\geq k-1$ for all $i\in [m-1]$; such a partition exists since $n-1 \geq (k-1)(m-1)$.

Now, consider the following collection of lines:
\begin{enumerate}[label=(\roman*)]
    \item the line $y = t_i$ for all $i\in [m-1]$;\label{item:bshorizontallines}
    \item $k-1$ copies of the line $x = s$ for all $s\in S_1\setminus\set{0}$; \label{item:bsverticallines}
    \item the line connecting $(0,t_i)$ and $(s,0)$ for every $i\in [m-1]$ and $s\in P_i$.\label{item:diagonallines}
\end{enumerate}

 In total, this collection contains $m-1+(k-1)(n-1) + n-1 = k(n-1) + m-1$ lines. It remains to verify that these lines form a valid $k$-cover of $\Gamma$. Note first that no line in this collection passes through the origin $(0,0)$.

Any interior point of $\Gamma$ is covered $k$ times by the lines in~\ref*{item:bshorizontallines} and~\ref*{item:bsverticallines}, leaving us to check the boundary points. A point of the form $(s,0)$, where $s\in S_1\setminus\set{0}$, is covered $k-1$ times by the lines in~\ref*{item:bsverticallines} and once by the lines in~\ref*{item:diagonallines}. Finally, a point $(0,s)$ for $s\in S_2\setminus\set{0}$ is covered once by the lines in~\ref*{item:bshorizontallines} and at least $k-1$ times by the lines in~\ref*{item:diagonallines} since each $P_i$ has size at least $k - 1$. Hence, every nonzero point of $\Gamma$ is covered at least $k$ times, as required.
\end{proof}

We remark that the above construction can be generalised to higher dimensions to show that, for any $n_1\times \dots \times n_d$ grid $\Gamma(S_1,\dots, S_d)$ containing the origin, if $n_1\geq n_2\geq \dots \geq n_d$ and $n_1\geq \cov[k-1](\Gamma(S_2,\dots, S_d))+1$, then the Ball--Serra bound is tight for $\cov(\Gamma(S_1,\dots,S_d))$. Indeed,
write $S_1 = \set{0,s_1,\dots, s_{n_1-1}}$ and let $\cH = \set{H_1,\dots, H_{n_1-1}}$ be any collection of $n_1-1$ hyperplanes  in $\mathbb{R}^{d-1}$ containing a $(k-1)$-cover of $\Gamma(S_2, \hdots, S_d)$. We then form a $k$-cover of $\Gamma(S_1, S_2, \hdots, S_d)$ consisting of the following hyperplanes in $\mathbb{R}^d$:
\begin{itemize}
    \item[(i)] one copy of the hyperplane $x_i = t$ for all $i\in \set{2,\dots, d}$ and $t\in S_i\setminus\set{0}$;
    \item[(ii)] $k-1$ copies of the hyperplane $x_1=s$ for all $s\in S_1\setminus\set{0}$;
    \item[(iii)] the hyperplane spanned by $\set{0}\times H_i$ and $(s_i,0,\dots, 0)$ for all $i\in [n_1-1]$.
\end{itemize}

It is not difficult to check that this is indeed a $k$-cover of $\Gamma(S_1,\dots, S_d)$, and it consists of $\sum_{i=1}^d n_i + (k-1)n_1$ hyperplanes, which matches the Ball--Serra lower bound~\eqref{eq:ball-serra-bound}. However, it is not clear how good the lower bound on $n_1$ is; that is, how large $n_1$ needs to be with respect to the other dimensions $n_i$ in order to ensure that the Ball--Serra bound is tight.

In our next result, we show that the bound on $n_1$ in Theorem~\ref{thm:asymmetric2d} is best possible, since for generic grids that are slightly less wide, the Ball--Serra bound is no longer tight. In fact, we a give a general lower bound for $\cov(\Gamma)$ when $\Gamma$ is a generic $n \times m$ grid, and prove that this bound is tight for infinitely many choices of $k$. While we will not pursue this question further for higher dimensions in this paper, we remark that it was shown in~\cite{bakker_2021_thesis} that for the grid $\set{0,\dots, n_1-1}\times\set{0,\dots, n_2-1}\times\set{0,\dots, n_3-1}$ with $n_1\geq n_2\geq n_3$, the Ball--Serra bound is already tight when $n_1\geq (k-1)(n_2-1)+1$, which is an improvement on the bound given by the above construction.

\begin{proof}[Proof of Theorem~\ref{thm:tightboundsasymmetric}]

We wish to show that if $\Gamma = \Gamma(S_1, S_2)$ is a generic grid, where $S_1, S_2 \subseteq \mathbb{R}$ satisfy $0 \in S_1 \cap S_2$ and $\size{S_1} = n \ge m = \size{S_2}$, then we have $\cov(\Gamma) \ge k(n-1) + \frac{k}{n+m-2} (m-1)^2$. Appealing to the linear programming framework developed in Section~\ref{sec:linearprogramming}, it suffices to show $\Phi(\Gamma) \ge (n-1) + \frac{(m-1)^2}{n+m-2}$, which can be done by defining a weighting on the nonzero points of $\Gamma$ with this total weight in which every line in $\cL$ has weight at most $1$.

To that end, define the weighting $w: \Gamma \setminus \set{(0,0)} \to \mathbb{R}$ by
\[
    w((x,y)) = \begin{cases}
        \frac{n-1}{n+m-2} & \textrm{if } y = 0; \\
        \frac{m-1}{n+m-2} & \textrm{if } x = 0; \\
        \frac{1}{n+m-2} & \textrm{otherwise}.
    \end{cases}
\]

We start by computing the total weight of the grid:
\begin{align*}
    w\parens*{\Gamma \setminus \set{(0,0)}} &= (n-1) \frac{n-1}{n+m-2} + (m-1) \frac{m-1}{n+m-2} + (n-1)(m-1) \frac{1}{n+m-2} \\
    &= (n-1)\frac{n-1+m-1}{n+m-2} + \frac{(m-1)^2}{n+m-2} \\
    &= n-1 + \frac{(m-1)^2}{n+m-2},
\end{align*}
and so, provided this weighting is feasible, it gives the desired lower bound.

To establish its feasibility, let us consider a line $\ell \in \cL(\Gamma)$. First suppose $\ell$ contains two boundary points, say $(x,0)$ and $(0,y)$. Since $\Gamma$ is generic, $\ell$ cannot contain any other points, and hence $w(\ell) = w((x,0)) + w((0,y)) = \frac{n-1}{n+m-2} + \frac{m-1}{n+m-2} = 1$. Next, suppose $\ell$ is a horizontal line of the form $y=s$ for some $s \in S_2 \setminus \set{0}$. The line $\ell$ then contains one point on the $y$-axis and $n-1$ interior points, and thus $w(\ell) = \frac{m-1}{n+m-1} + (n-1)\frac{1}{n+m-2} = 1$. Finally, any other line $\ell$ can contain at most one boundary point and at most $m-1$ interior points $(x,y)$, one for each choice of $y \in S_2 \setminus \set{0}$. For such lines, we therefore have $w(\ell) \le \frac{n-1}{n+m-2} + (m-1)\frac{1}{n+m-2} = 1$.

Hence, the weighting $w$ is indeed feasible for the dual linear program $\cD(\Gamma)$, and has total weight $n-1 + \frac{(m-1)^2}{n+m-2}$, which proves $\cov(\Gamma) \ge k(n-1) + \frac{k}{n+m-2} (m-1)^2$.

\medskip

Now, given an arbitrary $n \times m$ grid $\Gamma = \Gamma(S_1, S_2)$, we provide a construction of a $k$-cover of $\Gamma$ that matches the bound proven above for an infinite sequence of multiplicities $k$, thereby determining $\cov(\Gamma)$ for generic grids $\Gamma$ and such multiplicities $k$.

Defining $a = \frac{n-1}{\gcd(n-1,m-1)}$ and $b = \frac{m-1}{\gcd(n-1,m-1)}$, we are given that $k$ is divisible by $a+b$. We further define $d_1 = \frac{bk}{a+b}$ and $d_2 = \frac{ak}{a+b}$, noting that our divisibility assumption ensures these are integers and that $d_1+d_2 = k$. Let $B$ be an arbitrary biregular bipartite multigraph with parts $S_1 \setminus \set{0}$ and $S_2 \setminus \set{0}$ with degrees $d_1$ in the first part and $d_2$ in the second. Note that such a multigraph exists, since $d_1 (n-1) = d_2 (m-1)$, and we can assign $d_1$ half-edges to each $s_1 \in S_1 \setminus \set{0}$ and $d_2$ half-edges to each $s_2 \in S_2 \setminus \set{0}$, and then take an arbitrary matching between the two sets of half-edges.

Next, consider the  following collection of lines:
\begin{enumerate}[label=(\roman*)]
    \item $d_2$ copies of the line $x = s_1$ for each $s_1 \in S_1 \setminus \{0\}$;\label{coverasym:vertical}
    \item $d_1$ copies of the line $y = s_2$ for each $s_2 \in S_2 \setminus \{0\}$; \label{coverasym:horizontal}
    \item for each $\set{s_1, s_2} \in E(B)$, a copy of the line connecting $(s_1,0)$ to $(0,s_2)$.\label{coverasym:diagonal}
\end{enumerate}

To see that these lines form a $k$-cover of $\Gamma$, observe that every interior point is covered by $d_2$ vertical lines from~\ref{coverasym:vertical} and $d_1$ horizontal lines from~\ref{coverasym:horizontal}, and is thus covered $d_1 + d_2 = k$ times in total. For the boundary points, a point of the form $(s_1,0)$ for $s_1\in S_1\setminus\set{0}$ is covered $d_2$ times by the lines in~\ref{coverasym:vertical}, while the biregularity of the multigraph $B$ ensures it is covered $d_1$ times by the lines in~\ref{coverasym:diagonal}. Similarly, each point of the form $(0,s_2)$ for $s_2\in S_2\setminus\set{0}$ is covered $d_1$ times by the lines in~\ref{coverasym:horizontal} and $d_2$ times by those in~\ref{coverasym:diagonal}. Thus, the boundary points are also each covered $k$ times. Finally, none of the lines in our collection passes through the origin.

We thus obtain our upper bound by calculating the size of this cover, which yields
\begin{align*} 
    \cov(\Gamma) &\le d_2(n-1) + d_1(m-1) + d_1(n-1) \\
    &= (d_1 + d_2)(n-1) + d_1(m-1) \\
    &= k(n-1) + \frac{bk}{a+b}(m-1) \\
    &= k(n-1) + \frac{k}{n+m-2}(m-1)^2,
\end{align*}
as required.
\end{proof}

\section{Square grids}
\label{sec:squaregrids}

In the previous section, we saw that the Ball--Serra bound is tight when the grid is much wider than it is tall, and proved a lower bound for generic grids that becomes much larger than the Ball--Serra bound as the dimensions grow closer in size. Therefore, for the rest of this paper, we focus on $n \times n$ grids. In Section~\ref{sec:squaregeneral} we prove Theorem~\ref{thm:squaregridsgeneral}, which provides general lower and upper bounds on $\cov(\Gamma)$ for arbitrary $n \times n$ grids $\Gamma$. In Section~\ref{sec:squarestandard} we focus on the standard grid $\Gamma_n = \Gamma\parens*{\set{0,1,\hdots,n-1},\set{0,1,\hdots,n-1}}$, proving Theorems~\ref{thm:standardgridbounds} and~\ref{thm:lowerboundthreetypes}.

\subsection{General results} \label{sec:squaregeneral}

We start by restating our general bounds for square grids.

\squaregridsgeneral*

We will prove the upper bound of (a) via an explicit construction of a $k$-cover, and use the linear programming framework of Section~\ref{sec:linearprogramming} to establish the lower bounds of (b) and (c).

\begin{proof}
(a) Note that when $k$ is even, the upper bound $\cov(\Gamma) \le \tfrac32 k (n-1)$ follows from the upper bound in Theorem~\ref{thm:tightboundsasymmetric} when $m = n$. We will obtain the upper bound for odd $k$ with some appropriate rounding, and present a unified construction below.

Let $\set{x_1, x_2, \hdots, x_{n-1} }$ be the nonzero elements of $S_1$, and let $\set{y_1, y_2, \hdots, y_{n-1}}$ be the nonzero elements of $S_2$. We form a $k$-cover of $\Gamma$ consisting of the following lines:
\begin{enumerate}
    \item[(i)] $\ceil*{\frac12 k}$ copies of the lines $x = x_i$ and $y = y_i$, for each $i \in [n-1]$;
    \item[(ii)] $\floor*{\frac12 k}$ copies of line connecting $(x_i,0)$ to $(0,y_i)$, for each $i \in [n-1]$.
\end{enumerate}
There are $2\ceil*{\frac12 k} (n-1)$ lines in (i) and $\floor*{\frac12 k}$ lines in (ii), and hence we have a total of $\ceil*{\frac32 k} (n-1)$ lines, and it is evident that none of these pass through the origin. To see that they form a $k$-cover, observe first that each interior point is covered by $\ceil*{\frac12 k}$ horizontal lines and $\ceil*{\frac12 k}$ vertical lines, and is thus covered at least $k$ times in total. Meanwhile, each boundary point is covered $\ceil*{\frac12 k}$ times by the lines in (i) and $\floor*{\frac12 k}$ times by the lines in (ii), and so is incident to exactly $k$ lines. These lines therefore indeed form a $k$-cover, showing $\cov(\Gamma) \le \ceil*{\frac32 k}(n-1)$.

\medskip

(b) To prove lower bounds, we appeal to the dual linear program $\cD (\Gamma)$. Our goal is to define a weighting $w$ of the nonzero points of $\Gamma$ of large total weight in which no origin-avoiding line has weight more than $1$. The key observation is that these lines can contain at most two boundary points, and so we can hope to get away with assigning large weights to the boundary points. 

We shall first try a simple weighting $w'$, where all boundary points obtain a weight of $\alpha$, and all interior points a weight of $\beta$, for $\alpha$ and $\beta$ to be chosen later. Unfortunately, this initial attempt does not work. Indeed, we have $w'(\Gamma) = 2(n-1) \alpha + (n-1)^2 \beta = \parens*{2 \alpha + (n-1) \beta} (n-1)$. Since there could be lines containing two boundary points and $n-2$ interior points, we must have $2\alpha + (n-2) \beta \le 1$. Hence, $w'(\Gamma) \le (1 + \beta)(n-1)$. As lines parallel to the axes contain an boundary point and $n-1$ interior points, we must also have $\alpha + (n-1) \beta \le 1$, which in particular implies $\beta \le \frac{1}{n-1}$. Thus, $w'(\Gamma) \le n$, and so the best lower bound we can hope for from such a weighting is $\cov(\Gamma) \ge kn$, which is worse than the Ball--Serra bound~\eqref{eq:ball-serra-bound} for $n \ge k + 2$.

\medskip

To salvage this idea, we will instead only assign weight to some of the boundary points, with the aim of ensuring that any origin-avoiding line containing two positively-weighted boundary points cannot contain too many interior points. For this, we use the following claim, bounding the number of points contained in certain lines.

\begin{clm} \label{clm:linesize}
    Suppose we enumerate the members of $S_1$ as $x_1 < x_2 < \hdots < x_n$ and those of $S_2$ as $y_1 < y_2 < \hdots < y_n$, and suppose $i_0 \in [n]$ is such that $x_{i_0} = 0$. Let $j\in [n]$ and $\ell$ be a line passing through $(0,y_j)$. If $\ell$ has positive slope, then $\ell$ contains at most $n - |j - i_0|$ points of $\Gamma$. If $\ell$ has negative slope, then $\ell$ contains at most $n - |(n-j) - i_0|$ points of $\Gamma$.
\end{clm}

\begin{proof}[Proof of~\cref{clm:linesize}] Suppose first that $\ell$ is a line of positive slope passing through $(0,y_j)$. We define $S_1^- = \set{x \in S_1: x \le 0}$ and $S_2^- = \set{ y \in S_2: y \le y_j}$, observing that $\card{S_1^-} = i_0$ and $\card{S_2^-} = j$. Since $\ell$ is of positive slope, it follows that, if $(x,y) \in \ell$ for any $x \in S_1^-$, we must have $y \in S_2^-$.

If $j \ge i_0$, then, since each value in $S_1^-$ can correspond to at most one value in $S_2^-$, it follows that there will be at least $j - i_0$ coordinates in $S_2^-$, and thus in $S_2$, that are not mapped to by $\ell$, and so $\ell$ can contain at most $n - (j - i_0)$ points of $\Gamma$. Similarly, if $j < i_0$, since each value in $S_2^-$ corresponds to at most one value from $S_1^-$, there will be at least $i_0 - j$ values in $S_1$ not covered by $\ell$, and so $\ell$ contains at most $n - (i_0 - j)$ points from $\Gamma$. This shows that lines of positive slope through $(0,y_j)$ can contain at most $n - |j - i_0|$ points of $\Gamma$.

For lines of negative slope, we can reflect the grid in the $x$-axis, considering $S_2' = \set{ -y: y \in S_2}$. This reverses the ordering of the elements, so $y'_i = - y_{n-i}$. The line $\ell$ then corresponds to a line $\ell'$ of positive slope passing through $(0,y'_{n-j})$, and thus contains at most $n - |(n-j) - i_0|$ points of the grid.
\end{proof}

With this claim in mind, we now define an improved weighting of the points of $\Gamma$. As in the claim, enumerate the elements of $S_1$ as $x_1 < x_2 < \hdots < x_n$ and those of $S_2$ as $y_1 < y_2 < \hdots < y_n$, and let $i_0, j_0 \in [n]$ be such that $x_{i_0} = 0$ and $y_{j_0} = 0$. Given parameters $\alpha, \beta$ and $t$, to be chosen later, we define the weights on $\Gamma \setminus \set{ (0,0) }$ as follows:
\[ w( (x_i,y_j) ) = 
    \begin{cases}
        \alpha & \textrm{if } i \neq i_0 \textrm{ and } j \neq j_0; \\
        \beta & \textrm{if } j = j_0 \textrm{, or if } i = i_0 \textrm{ and } \min \set{|j-i_0|,|n-j-i_0|} \ge t; \\
        0 & \textrm{otherwise}.
    \end{cases}        
\]
That is, we assign weight $\alpha$ to all interior points and weight $\beta$ to all boundary points except those in intervals around $y_{i_0}$ and $y_{n-i_0}$ on the $y$-axis.

For the weighting to be valid, we require that each origin-avoiding line have weight at most $1$. If the line $\ell$ contains two boundary points with positive weight, then let it pass through $(0,y_j)$ on the $y$-axis. By definition of $w$, we must have $|j-i_0|,|(n-j) - i_0| \geq t$, and so by Claim~\ref{clm:linesize}, $\ell$ contains at most $n-t$ points of $\Gamma$, and hence at most $n-t-2$ interior points. Thus, the weight of any such line is at most $2 \beta + (n-t-2) \alpha$.

Otherwise, the line $\ell$ can contain at most one weighted boundary point and at most $n-1$ interior points, giving a total weight of not more than $\beta + (n-1) \alpha$. Hence, our parameters must satisfy  $2 \beta + (n-t-2) \alpha \le 1$, $\beta + (n-1) \alpha \le 1$, $\alpha, \beta \ge 0$, and $t \in \mathbb{N}$.

With regards to the objective function, we note that there are at most $2(2t-1)$ boundary points with weight zero, and thus $w(\Gamma) \ge 2(n - 2t) \beta + (n-1)^2 \alpha$, and we wish to maximise this quantity subject to the constraints above. Some routine calculations then yield that we should set $t = \ceil*{\tfrac12 \sqrt{(5n+1)(n-1)} - n}$, $\alpha = \frac{1}{n+t}$ and $\beta = \frac{t+1}{n+t}$, for which we have $w(\Gamma) = \left( 10 - 4 \sqrt{5} + o(1) \right) (n-1)$. Proposition~\ref{prop:linprogbound} gives the desired lower bound:
\[ \cov(\Gamma) \ge k \Phi(\Gamma) \ge k w(\Gamma) = \left( 10 - 4 \sqrt{5} + o(1) \right) k (n-1). \]

\medskip

(c) For the final part of the theorem, we assume the grid $\Gamma$ is $\Delta$-generic, meaning that any line through two boundary points can contain at most $\Delta$ interior points. In the framework of part (b), where we assign a weight of $\alpha$ to all interior points, and a weight of $\beta$ to all boundary points, we then obtain the constraint $\beta + (n-1)\alpha \le 1$ from the lines with at most one boundary point, and the constraint $2 \beta + \Delta \alpha \le 1$ from the lines with two boundary points. The total weight, which we seek to maximise, is $2(n-1)\beta + (n-1)^2 \alpha$.

This optimisation problem is solved by taking $\beta = 1 - \frac{n-1}{2(n-1)-\Delta}$ and $\alpha = \frac{1}{2(n-1) - \Delta}$. It is readily verified that both constraints are then satisfied with equality, and the total weight of the grid is $\left[ 2 - \frac{n-1}{2(n-1) - \Delta} \right] (n-1)$, whence the result follows by once again applying Proposition~\ref{prop:linprogbound}.
\end{proof}

A few remarks are in order at this point. First, we note that the lower bound of part~\ref{thm:squaregridslb} can be improved if we have additional information about where the origin is in the grid. For example, suppose the origin is in the lower-left corner; that is, $\min S_1 = \min S_2 = 0$. Then any line containing two boundary points must be of negative slope, and hence when we apply Claim~\ref{clm:linesize}, we see that it is enough to leave only the largest values on the $y$-axis unweighted. When one solves the corresponding optimisation problem, we find that $\cov(\Gamma) \ge \left( 4 - 2 \sqrt{2} + o(1) \right) k (n-1)$ for such grids $\Gamma$, a considerable improvement in the constant factor, as $4-2\sqrt{2}\approx 1.1716$.

Second, as explained in the introduction, almost all grids $\Gamma$ are generic, and parts~\ref{thm:squaregridsub} and~\ref{thm:squaregridsgeneric} of Theorem~\ref{thm:squaregridsgeneral} determine $\cov(\Gamma)$ precisely when $k$ is even and asymptotically when $k$ is odd and large. However, even when the grid $\Gamma$ is not generic but only $\Delta$-generic, provided $\Delta = o(n)$, part~\ref{thm:squaregridsgeneric} is robust enough to resolve the problem asymptotically. We give some natural examples of this below.

\begin{cor}
    Given $n \in \mathbb{N}$, let $\Gamma_{\mathrm{exp},n} = \Gamma(E,E)$, where $E = \set{0,1,2,4,\hdots,2^{n-2}}$, and let $\Gamma_{\mathrm{quad},n} = \Gamma(S,S)$, where $S = \set{0,1,4,\hdots, (n-1)^2}$. Then, if $\Gamma \in \set{ \Gamma_{\mathrm{exp},n}, \Gamma_{\mathrm{quad},n} }$, we have $\cov(\Gamma) = \left( \frac32 + o(1) \right) k(n-1)$ as $k,n \to \infty$.
\end{cor}

\begin{proof}
    By Theorem~\ref{thm:squaregridsgeneral}\ref{thm:squaregridsub}, we know $\cov(\Gamma) \le \ceil*{\frac32k} (n-1)$, which is $\left(\frac32 + o(1)\right) k(n-1)$ as $k \to \infty$, and so we need only demonstrate the lower bound. By Theorem~\ref{thm:squaregridsgeneral}\ref{thm:squaregridsgeneric}, it suffices to show $\Gamma$ is $\Delta$-generic for some $\Delta = o(n)$.

    We begin with $\Gamma = \Gamma_{\mathrm{exp},n}$, and show that is $1$-generic. Indeed, suppose $\ell$ is a line containing two boundary points of $\Gamma$, say $(0,2^i)$ and $(2^j,0)$. It is then a straightforward calculation to see that, for any $r$, the line $\ell$ passes through $(2^r, 2^i - 2^{i-j+r})$. In order for this to be an interior point of the grid $\Gamma_{\mathrm{exp},n}$, we require that $2^i - 2^{i-j+r}$ is a positive power of $2$ strictly smaller than $2^i$, which only happens for $r = j-1$. Thus, the line $\ell$ contains exactly one interior point, and hence $\Gamma_{\mathrm{exp},n}$ is $1$-generic.

    \medskip
    
    The quadratic grid $\Gamma_{\mathrm{quad},n}$ requires somewhat more delicate treatment. Again, let us suppose $\ell$ is a line containing the boundary points $(0,a^2)$ and $(b^2,0)$, and let $(r^2, s^2)$ be an interior point lying on $\ell$. We then have $s^2 = a^2 \parens*{1 - \frac{r^2}{b^2}}$, or $(ab)^2 = (ar)^2 + (bs)^2$. Thus, we can bound the number of interior points on $\ell$ by the number of ways of writing $(ab)^2$ as a sum of squares. Following Beiler~\cite{Bei66}, if $Q$ is the set of prime divisors of $ab$ that are congruent to $1$ modulo $4$, and if $\beta_q$ is the multiplicity of $q \in Q$ in the prime factorisation of $(ab)^2$, then there are $\prod_{q \in Q} \parens*{\beta_q + 1}$ ways to write $(ab)^2$ as a sum of squares.
    
    To simplify the notation in the calculation below, we shall assume $q = 5, 13$ and $17$ are included in $Q$, setting $\beta_q = 0$ in case they do not divide $ab$. Now observe that we have 
    \begin{align*} 
    \sum_{q \in Q} \parens*{\beta_q + 1} &= 3 +  \beta_5 + \beta_{13} + \beta_{17} + \sum_{q \in Q, q \ge 29} \parens*{\beta_q + 1} \\ 
    &\le 3 + \beta_5 + \beta_{13} + \beta_{17} + \sum_{q \in Q, q \ge 29} 2 \beta_q \\
    &\le 3 + \sum_{q \in Q} \beta_q \log_5 q = 3 + \log_5 \parens*{\prod_{q \in Q} q^{\beta_q}} \le 3 + \log_5 \parens*{(ab)^2},
    \end{align*}
    and so $\sum_{q \in Q} (\beta_q + 1) \le 3 + 4 \log_5 n$.

    Now, given some natural numbers $m_i$, it is simple to show that $\prod_i m_i \le 3^{\tfrac13 \sum_i m_i}$, and hence the number of ways to express $(ab)^2$ as a sum of squares is at most $3^{\tfrac13 \sum_{q \in Q} (\beta_q + 1)} \le 3 \cdot 3^{\tfrac43 \log_5 n} = 3n^{\tfrac43 \log_5 3}$. It thus follows that $\Gamma_{\mathrm{quad},n}$ is $\Delta$-generic for $\Delta = 3n^{\tfrac43 \log_5 3} = o(n)$.
\end{proof}

\subsection{Standard grids} \label{sec:squarestandard}

While the results of Section~\ref{sec:squaregeneral} resolve the problem asymptotically for very many grids, there is no questioning the fact that the most natural case to consider is that of the standard grid $\Gamma(S, S)$, where $S = \set{0, 1, \hdots, n-1}$. For convenience, we denote this grid by $\Gamma_n$. By considering the line $x + y = n-1$, we see that $\Gamma_n$ is not $\Delta$-generic for any $\Delta < n-2$, which means the lower bound of Theorem~\ref{thm:squaregridsgeneral}\ref{thm:squaregridsgeneric} is worse than the Ball--Serra bound for $n > k$. By tailoring our methods for this specific grid, though, we will obtain much better bounds. We begin by showing that a strict $k$-cover of the standard grid requires far fewer lines than the upper bound given by~\cref{thm:squaregridsgeneral}, which, as shown in the previous subsection, is asymptotically tight for most grids.

\begin{proof}[Proof of~\cref{thm:standardgridbounds} (upper bound)]
   We construct a $k$-cover of $\Gamma_n$ of size $\parens{\sqrt{2}+o(1)}k(n-1)$. Let $t\in [n-1]$ be a parameter, to be determined later, and consider the following collection of lines:
   \begin{enumerate}[label=(\roman*)]
        \item $\ceil*{\frac{i}{n+t-1}k}$ copies of the lines $x=i$ and $y=i$ for each $i\in [n-1]$;\label{list:standard_ub_par}
        \item $k - \ceil*{\frac{i}{n+t-1}k}$ copies of the line $x+y=i$ for every $1\leq i < n+t-1$. \label{list:standard_ub_diag}
   \end{enumerate}

   We begin by showing that the above collection of lines gives a $k$-cover of $\Gamma_n$. First, it is clear that no line passes through the origin. Now consider a point $(s_1,s_2)\in \Gamma_n\setminus\set{(0,0)}$. 
   If $s_1+s_2< n+t-1$, then $(s_1,s_2)$ is covered $\ceil*{\frac{s_1}{n+t-1}k} + \ceil*{\frac{s_2}{n+t-1}k}$ times by the lines in~\ref*{list:standard_ub_par} and another $k-\ceil*{\frac{s_1+s_2}{n+t-1}k}$ times by those in~\ref*{list:standard_ub_diag}, and thus at least $k$ times in total. On the other hand, if $s_1+s_2\geq n+t-1$, then  lines in~\ref*{list:standard_ub_par} alone cover the point $\ceil*{\frac{s_1}{n+t-1}k} +\ceil*{\frac{s_2}{n+t-1}k} \geq k$ times, as required. Calculating the size of this $k$-cover, we obtain 
   \begin{align}
       \cov(\Gamma_n) &\leq 2\sum\limits_{i=1}^{n-1} \ceil*{\frac{i}{n+t-1}k}  + \sum\limits_{i=1}^{n+t-2}\parens*{k-\ceil*{\frac{i}{n+t-1}k}} \notag \\
       &\leq k\brackets*{2\sum\limits_{i=1}^{n-1} {\frac{i}{n+t-1}}  + \sum\limits_{j=1}^{n+t-2}{\frac{j}{n+t-1}}} + 2n \notag\\
       &\leq k\brackets*{ \frac{2}{n+t-1}\binom{n}{2}  + \frac{1}{n+t-1}\binom{n+t-1}{2} } + 2n \notag\\
       &= k \brackets*{\frac{n(n-1)}{n+t-1} + \frac{n+t-2}{2}} + 2n.\label{eq:totalweight}
   \end{align}
   
   The upper bound given by~\eqref{eq:totalweight} is valid for any $t\in [n-1]$; we now want to choose a value of  $t$ that makes the right-hand side as small as possible. The function $g(t) = \frac{n(n-1)}{n+t-1} + \frac{n+t-2}{2}$ has its minimum at $t_0 = \sqrt{2(n-1)n} - (n-1)$.
   Since our parameter must be an integer, we choose $t = \ceil*{\sqrt{2(n-1)n} - (n-1)} = (\sqrt{2}-1+o(1))(n-1)$, and substituting this value of $t$ into~\eqref{eq:totalweight} yields the claimed upper bound of  $\parens{\sqrt{2}+o(1)}k(n-1)$ on $\cov(\Gamma_n)$. 
\end{proof}

    We now turn our attention to the lower bound.

    \begin{proof}[Proof of~\cref{thm:standardgridbounds} (lower bound)]
    By~\cref{prop:linprogbound}, it suffices to find a feasible solution to the dual linear program $\cD(\Gamma_n)$, that is, a weighting of the nonzero points of $\Gamma_n$ in which every origin-avoiding line has weight at most 1, that has total weight at least $(2 - e^{-1/2} +o(1)) (n-1)$.
    
    Let $t$ be the largest integer such that $\sum_{i=1}^t \frac{1}{n-i} \le \tfrac12$ and consider the following weighting on the points of $\Gamma_n\setminus\set{(0,0)}$:
    \[
    w((x,y)) = \begin{cases}
        \frac12 & \textrm{if } xy = 0; \\
        \frac{1}{n-i} & \textrm{if } x + y =  n-1+i \textrm{ for some } i\in [t]; \\
        0 & \textrm{otherwise}.
    \end{cases}
    \]

    We first show that $(w((x,y)):(x,y)\in \Gamma_n)$ gives a feasible solution to the dual linear program $\cD(\Gamma_n)$. Clearly $w((x,y)) \geq 0$ for all $(x,y)\in \Gamma_n\setminus\set{(0,0)}$. Now, let $\ell\in \mathcal{L}$ be any line. If $\ell$ contains two boundary points, then any interior point $(x,y)$ on $\ell$ satisfies $x + y \le n-1$, and thus has weight zero. It follows that $w(\ell) = 1$. Otherwise, if $(x,y) \mapsto x+y$ is constant on $\ell$, then $\ell = \{(x,y): x+y = n-1+i\}$ for some $i \in [n-1]$. Then all points on $\ell$ have weight zero, unless $1 \le i \le t$, in which case $\ell$ contains $n-i$ points of weight $\tfrac{1}{n-i}$ each. Thus $w(\ell) \le 1$ in this case. Finally, if $(x,y) \mapsto x+y$ is not constant on $\ell$, it must be injective. Then, $\ell$ contains at most one boundary point, which has weight $\frac12$, and the weight from the remaining points is at most $\sum_{i=1}^t \tfrac{1}{n-i}$, which by the choice of $t$ is at most $\tfrac12$. So in total we again have $w(\ell)\leq 1$.

    \medskip

    To compute the total weight of the grid, observe that each diagonal line of the form $x + y = i$ has weight one if $1 \le i \le n-1+t$ and zero otherwise. Thus, the total weight of the grid is $n-1+t$.

    It remains to estimate $t$. Note that $t \le \tfrac{n}{2}$, since $\sum_{i=1}^t \tfrac{1}{n-i} \ge \sum_{i=1}^t \tfrac{1}{n} = \tfrac{t}{n}$, and so both $n-1$ and $n-1-t$ go to infinity linearly with $n$.
    It is well known that, as $m\to\infty$, the partial sums $H_m$ of the Harmonic series satisfy $H_m = \sum_{j=1}^m \tfrac{1}{j} = \log m + \gamma + o(1)$, where $\gamma$ is a constant.  Hence,
    \begin{align*}
        \sum_{i=1}^t \frac{1}{n-i} &= H_{n-1} - H_{n-1-t}
        = \log \left( \frac{n-1}{n-1-t} \right) + o(1).
    \end{align*}

    Thus, we must have $\log \left( \frac{n-1}{n-1-t} \right) = \tfrac12 + o(1)$, or $\log \left( 1 - \frac{t}{n-1} \right) = -\tfrac{1}{2} + o(1)$. This gives $1 - \tfrac{t}{n-1} = e^{-1/2} + o(1)$, or $t = \left( 1 - e^{-1/2} + o(1) \right) (n-1)$, which results in the claimed bound.
    \end{proof}

The cover we constructed for the upper bound only uses lines of slope $0$, $\infty$, and $-1$. This may seem rather limited, and it is natural to wonder if one can do better by making use of other lines as well. However, we verified by computer search that for small values of $n$ and $k$, one is always able to build an optimal $k$-cover consisting only of these three types of lines. This motivated us to further study this restricted class of $k$-covers, and in our final result we prove that the smallest $k$-cover of this form has size $\parens*{\sqrt{2} + o(1)} k(n-1)$.
    
    \begin{proof}[Proof of~\cref{thm:lowerboundthreetypes}]
    For convenience, we will hereon call lines of slope $-1$ \emph{diagonals}. Since the cover constructed for the upper bound in~\cref{thm:standardgridbounds} only uses horizontal, vertical, and diagonal lines, it also provides the same upper bound of $\parens*{\sqrt{2} + o(1)}k(n-1)$ in this restricted setting.
    
    To obtain a matching lower bound, we will again appeal to the linear programming approach. Recall that previously we obtained lower bounds by assigning weights to the points, whose sum was as large as possible, provided that the total weight along every origin-avoiding line was at most $1$. In this restricted setting, since we are only able to use horizontal, vertical, and diagonal lines, the dual linear program only has constraints on the weights along these lines. This gives us much more freedom in choosing the weights, and thus we may hope to find a feasible weighting with a larger sum.

    Observe that in our search for an optimal weighting, we have $n^2 - 1$ degrees of freedom (the weights of the individual nonzero points) but only $4(n-1)$ constraints (the horizontal, vertical, and diagonal lines). In order to reduce the search space and simplify our task, we shall impose the following additional conditions on the weighting $w$.
    \begin{itemize}
    \item The weighting is symmetric across the main diagonal; that is, $w(x,y) = w(y,x)$ for all $(x,y)\in \Gamma_n\setminus\set{(0,0)}$.
    \item On each diagonal, every interior point has the same weight.
    \item Every vertical line avoiding the origin has weight one (and hence so does every horizontal line).
    \item There is some $t \in [n-1]$ such that the weight of the diagonal $x + y = i$ is $1$ if $1 \le i \le n + t - 1$, is at most $1$ if $i = n+t$, and is $0$ if $i \ge n + t + 1$.
    \end{itemize}

    Some remarks are in order now. First, note that the requirement of symmetry is without loss of generality, since if $w$ is any feasible weighting, then $w'$ defined by $w'(x,y) = \tfrac12 \parens*{w(x,y) + w(y,x)}$ is a symmetric feasible weighting of the same total weight.

    With regards to the total weight, by summing along the diagonals, we see that $w\parens{\Gamma_n \setminus \set{(0,0)}}$ is at least $n+t-1$ and at most $n+t$. Thus, our goal is to maximise the value of $t$ for which we can find such a feasible weighting.

    Now, for some fixed $t$, observe that the weights of approximately half the points are already determined by the conditions. Since the diagonals $x + y = i$, for $i \ge n + t + 1$, have weight $0$, we must have $w(x,y) = 0$ whenever $x + y > n + t$. On the other hand, if $n \le i \le n + t - 1$, then we know the diagonal has weight $1$. Since the diagonal consists entirely of $2n - 1 - i$ internal points, all of which must have the same weight, we have $w(x,y) = \frac{1}{2n - 1 - (x+y)}$ whenever $n \le x + y \le n+t-1$. Finally, when $x + y = n+t$, we must have $w(x,y) = z$ for some $0 \le z \le \frac{1}{n - t - 1}$ to ensure the diagonal has weight at most $1$.

    We now turn our attention to the lower triangular points of the grid; that is, $(x,y)$ with $1 \le x + y \le n-1$. In our previous weighting, we assigned weight $\tfrac12$ to the boundary points and left the interior points unweighted. Now that we have some more freedom, we will look to spread the weight throughout the grid. With that in mind, we introduce parameters $\alpha_i$, $1 \le i \le n-1$, such that $w(i,0) = w(0,i) = \tfrac12 - \alpha_i$. Since the diagonal $x + y = i$ has total weight $1$, it follows that $w(x,y) = \frac{2 \alpha_{x + y}}{x + y - 1}$ for the $x + y - 1$ interior points $(x,y)$ with $x + y = i$, $x,y \neq 0$.

    Thus, using the symmetry and the conditions along the diagonals, we have shown that our weighting takes the form
    \[
    w((x,y)) = \begin{cases}
        \frac{1}{2} - \alpha_{x+y} & \textrm{if } x=0 \textrm{ or } y = 0; \\
        \frac{2\alpha_{x+y}}{x+y-1} & \textrm{if } x,y \neq 0 \textrm{ and } 1\leq x+y\leq n-1; \\
         \frac{1}{2n-1-i} & \textrm{if } x+y = i \textrm{ for some } n\leq i\leq n+t-1; \\
         z & \textrm{if } x+y = n+t; \\
         0 & \textrm{if } x+y \ge n+t+1.
    \end{cases}
    \]
    for some parameters $\alpha_1, \hdots, \alpha_{n-1}, z \in \mathbb{R}_{\ge 0}$. To finish, we will use the condition that the vertical lines have total weight $1$ to solve for $\alpha_i$.

    Indeed, by considering the line $x = n-1$, we have $\tfrac12 - \alpha_{n-1} + \sum_{i=n}^{n+t-1} \frac{1}{2n-1-i} + z = 1$, which yields
    \begin{equation}\label{eq:constrn-1}
        \alpha_{n-1} = \sum\limits_{i=n}^{n+t-1} \frac{1}{2n-1-i} + z -\frac{1}{2}.
    \end{equation}

    Now compare the weights of the points on the lines $x = n-2$ and $x = n-1$. Since the diagonals are constant along their interior points, we have $w(n-2,y) = w(n-1,y-1)$ for all $2 \le y \le n-1$. Hence the differences are that $w(n-2,0)$ and $w(n-2,1)$ replace $w(n-1,0)$ and $w(n-1,n-1)$. Thus, 
    \begin{align*}
        w\parens*{\set*{x = n-2}} - w\parens*{\set*{x = n-1}} &= w(n-2,0) + w(n-2,1) - w(n-1,0) - w(n-1,n-1) \\
        &= \tfrac12 - \alpha_{n-2} + \frac{2 \alpha_{n-1}}{n-2} - \parens*{\tfrac12 - \alpha_{n-1}} - 0,
    \end{align*}
    and since both vertical lines have weight $1$, this gives $\alpha_{n-2} = \parens*{1 + \frac{2}{n-2}} \alpha_{n-1}$. Repeating this argument for the lines $x = i-1$ and $x = i$ for each $2 \le i \le n-1$, we obtain the following recurrence relation:
    
    \begin{equation}\label{eq:alpharecursion}
        \alpha_{i-1} = \parens*{1+\frac{2}{i-1}}\alpha_{i} - z\mathbf{1}_{i=t+1} -\frac{1}{n-i} \mathbf{1}_{i\leq t},
    \end{equation}
    where $\mathbf{1}_{A}$ is the indicator function of the event $A$ defined as
    \[
    \mathbf{1}_{A} = \begin{cases}
        1 & \textrm{if $A$ is true}; \\
        0 & \textrm{if $A$ is false}. \\
    \end{cases}
    \]

    For the initial condition, observe that the line $x + y = 1$ has no interior points, and so for it to have weight $1$, we must have $\alpha_1 = 0$. Combining this with~\eqref{eq:alpharecursion}, we obtain:
    \begin{equation}
        \alpha_{i} = \frac{1}{i(i+1)}\sum\limits_{j=1}^{\min\set{t, i}} \frac{j(j-1)}{n-j} + \mathbf{1}_{i\geq t+1}\frac{t(t+1)}{i(i+1)}z  \quad \text{ for all } 2\leq i\leq n-1.\label{eq:alphavalue}
    \end{equation}
    Substituting the value of $\alpha_{n-1}$ from~\eqref{eq:alphavalue} into~\eqref{eq:constrn-1}, we can then solve for $z$ to obtain:
    \begin{align}\label{eq:zvalue}
        z = \parens*{\frac{1}{2} - \sum\limits_{j=1}^t\frac{1}{n-j} \parens*{1-\frac{j(j-1)}{(n-1)n}}}\frac{n(n-1)}{n(n-1)-t(t+1)}\\
        = \frac{(n-1)n\parens*{\frac{1}{2} - \frac{t(2n+t-1)}{2(n-1)n} }}{(n-1)n - t(t+1)},
    \end{align}
    where the second equality is due to the fact that
    \begin{equation*}
        \sum\limits_{j=1}^t\frac{1}{n-j} \parens*{1-\frac{j(j-1)}{(n-1)n}} = \frac{1}{n(n-1)}\sum\limits_{j=1}^t(n+j-1) = \frac{t(2n+t-1)}{2n(n-1)}. 
    \end{equation*}
    
    Recall that feasibility dictates $0\leq z\leq \frac{1}{n-t-1}$. If $n>1$,  we have $z\geq 0$  when $0\leq t\leq  \frac{1}{2} \parens*{\sqrt{8 n^2 - 8 n + 1} - 2 n + 1}$ and $z\leq \frac{1}{n-t-1}$ for $\frac{1}{2} \parens*{\sqrt{8 n^2 - 8 n + 1} - 2 n - 1}\leq t < n - 1$. Taking $t$ to be an integer satisfying $\frac{1}{2} \parens*{\sqrt{8 n^2 - 8 n + 1} - 2 n - 1} \leq t \leq \frac{1}{2} \parens*{\sqrt{8 n^2 - 8 n + 1} - 2 n + 1}$, we have $t = (\sqrt{2}-1+o(1))(n-1)$. It follows that the total weight of the grid is $\parens*{\sqrt{2} +o(1)}(n-1)$.

    We are not quite done, as there is one final condition to verify --- to ensure that all our weights are non-negative, we must have $0\leq \alpha_i\leq \frac{1}{2}$ for all $1\leq i\leq n-1$. From~\eqref{eq:alpharecursion} we have:
    \begin{align}
        \alpha_i &= \frac{i-1}{i+1}\parens*{\alpha_{i-1}+\frac{1}{n-i}} & \text{if } 2\leq i\leq t.\notag \\
        \alpha_{i} &=\frac{i-1}{i+1}(\alpha_{i-1}+z)\leq \frac{i-1}{i+1}\parens*{\alpha_{i-1}+\frac{1}{n-i}} & \text{if } i=t+1\notag\\
        \alpha_i &= \frac{i-1}{i+1}\alpha_{i-1} <\alpha_{i-1} & \text{if } t+2\leq i\leq n-1\notag
    \end{align}

    Thus, it suffices to show that $\alpha_i\leq \frac{1}{2}$ for all $2\leq i\leq t+1$. We will do so by showing that for $1\leq i\leq t+1$, we have $\alpha_i \leq \frac{i-1}{2(n-i-1)}$ by induction on $i$. We know that $\alpha_1=0$, so the base case is clear. Let $i>1$ and assume the induction hypothesis; then
    \begin{equation*}
        \alpha_i \leq \frac{i-1}{i+1}\parens*{\alpha_{i-1}+\frac{1}{n-i}} \leq \frac{i-1}{i+1}\parens*{\frac{i-2}{2(n-i)} +\frac{1}{n-i}} \le \frac{i-1}{2(n-i-1)}.
    \end{equation*}
    We have $\frac{i-1}{2(n-i-1)}\leq\frac{1}{2}$ whenever $i\leq n-i$, which is true since $i\leq t+1 = (\sqrt{2}-1+o(1))(n-1)$.

    Hence our weighting is indeed feasible and $w\parens*{\Gamma_n \setminus \set*{(0,0)}} = \parens*{\sqrt{2} + o(1)} (n-1)$. It thus follows that any $k$-cover of $\Gamma_n$ using only horizontal, vertical, and diagonal lines must have size at least $\parens*{\sqrt{2} + o(1)} k(n-1)$.
    \end{proof}
    
\section{Conclusion} \label{sec:conclusion}

In this paper, we studied line coverings with multiplicities for two-dimensional real grids. We determined the minimum size of a cover in several cases, but some natural and interesting questions remain open, and we highlight them below. 

\medskip
In Section~\ref{sec:ballserra}, we investigated for which grids the Ball--Serra bound is tight. We proved that, when $n$ is sufficiently large with respect to $m$ and $k$, the Ball--Serra bound is tight for any $n\times m$ grid. Moreover, we showed that the threshold value for $n$ given by~\cref{thm:asymmetric2d} is tight for \emph{most} grids. It can be shown, however, that this bound on $n$ is not best possible for \emph{all} grids. For example, for the grid $\Gamma(S_1,S_2)$, where $S_1 = \set{0,1,2,\dots, n-1}$ and $S_2 = \set{-1,0,1}$, and any $k\geq 3$, we can show that the Ball--Serra bound is tight already for $n = 2(k-1)$, as opposed to the lower bound $n \ge 2k-1$ of~\cref{thm:asymmetric2d}. However, in this grid the omitted point $(0, 0)$ is not a corner point, while, as in the square grid setting, it is more natural to consider grids in which $(0,0)$ is a corner. For such grids, one could investigate when the Ball--Serra bound holds.

\begin{ques}
    Let $\Gamma$ be the grid $\set{0,1,2,\dots, n-1}\times \set{0,1,2,\dots, m-1}$ and $k\geq 2$ be an integer. How large must $n$ be with respect to $m$ and $k$ to have $cov(\Gamma) = k(n-1) + (m-1)$?
\end{ques}

Our main result for standard grids establishes reasonably good asymptotic lower and upper bounds on $\cov(\Gamma_n)$. It would be of interest to close the remaining gap. 

\begin{ques}
    What is the true asymptotic value of $\cov(\Gamma_n)$?
\end{ques}

We tend to believe that $\cov(\Gamma_n) = (\sqrt{2}+o(1))k(n-1)$. In~\cref{thm:lowerboundthreetypes}, we showed this to be the case when we only use lines of slope $0$, $\infty$, and $-1$. However, for the weighting we used to establish the lower bound, one can show that lines of slope $1$ near the origin (e.g., $y = x+1$) have weight larger than $1$ when $n$ is large. We believe that these are the only problematic lines, and so as an intermediate step one could attempt to verify that our weighting remains feasible if one only forbids lines of slope $1$. This would imply that any $k$-cover of $\Gamma_n$ of size smaller than $\parens*{\sqrt{2} + o(1)} k(n-1)$ must contain many lines of slope $1$. To show that such a construction is unlikely to exist, it might be helpful to consider what happens if we restrict ourselves to lines of slope $0$, $\infty$, $-1$, and $1$.

\medskip
In our work thus far we observed that the standard grid $\Gamma_n$ requires many fewer lines to cover than any other $n\times n$ grid we considered. Our general lower bound from~\cref{thm:squaregridsgeneral}\ref{thm:squaregridslb} (and the improvement for grids in which $(0,0)$ is a corner discussed after the proof) is not strong enough to establish this fact, and we propose the following problem.

\begin{ques}
    Is it true that $\cov(\Gamma_n) \leq \cov(\Gamma)$ for any $n\times n$ grid $\Gamma$ in which $(0,0)$ is a corner?
\end{ques}
More broadly, it would be of interest to improve the lower bound from~\cref{thm:squaregridsgeneral}\ref{thm:squaregridslb}, which we do not believe to be best possible. 

\medskip
Another direction that might lead to interesting findings is to consider translates of the standard grid in which the omitted point is not in the lower-left corner. How does the position of the origin then affect the value of $\cov(\Gamma)$? For instance, what are the asymptotics of $\cov(\Gamma)$, where $\Gamma =\Gamma(\set{-\floor{n/2},\dots, \ceil{n/2}},\set{-\floor{n/2},\dots, \ceil{n/2}})$?

\medskip
While we mainly focused on two-dimensional real grids, it would be natural to investigate the problem in higher dimensions as well. For example, in Section~\ref{sec:ballserra}, we remarked that the Ball--Serra bound can be tight for higher-dimensional grids as well, provided that one of the sides is much longer than the others. How much longer does that side need to be for the bound to be attained? Some first results in this direction were shown in the bachelor's thesis of the fourth author~\cite{bakker_2021_thesis}. Once again, it would be particularly interesting to investigate $\cov\parens{\Gamma_n^{(d)}}$ for the standard $d$-dimensional grid $\Gamma_n^{(d)} = \set{0,\dots, n-1}^d$.

\medskip
Finally, while all of our results are stated for grids over $\mathbb{R}$, the questions we considered and our general framework extend to grids over any field. Some of our results, for example Theorem~\ref{thm:asymmetric2d}, extend to arbitrary fields. 
It will be interesting to prove similar results for other fields, and in particular for fields of positive characteristic.

\end{document}